\documentclass{amsart}
\usepackage{amsfonts}
\usepackage{amsmath,amssymb}
\usepackage{amsthm}
\usepackage{amscd}
\usepackage{graphics}
\usepackage{graphicx}
\theoremstyle{definition}{
	\newtheorem{Def}{{\rm Definition}}
	\newtheorem{Ex}{{\rm Example}}
	\newtheorem{Rem}{{\rm Remark}}
	\newtheorem{Prob}{{\rm Problem}}
}
\theoremstyle{plain}
{
	\newtheorem{Cor}{Corollary}
	\newtheorem{Prop}{Proposition}
	\newtheorem{Thm}{Theorem}
	\newtheorem{MainThm}{Main Theorem}

}

\begin{document}
	\title[Smooth maps like special generic maps]{Smooth maps like special generic maps}
	\author{Naoki Kitazawa}
	\address{Institute of Mathematics for Industry, Kyushu University, 744 Motooka, Nishi-ku Fukuoka 819-0395, Japan\\
		TEL (Office): +81-92-802-4402 \\
		FAX (Office): +81-92-802-4405 \\
	}
	\email{n-kitazawa@imi.kyushu-u.ac.jp}
	\urladdr{https://naokikitazawa.github.io/NaokiKitazawa.html}
	
	\begin{abstract}
		In our paper, we introduce {\it special-generic-like} maps or {\it SGL} maps as smooth maps, present their fundamental algebraic topological and differential topological theory and give non-trivial examples. 
		
		The new class generalize the class of so-called {\it special generic} maps. Special generic maps are smooth maps which are locally projections or the product maps of Morse functions and the identity maps on disks. Morse functions with exactly two singular points on spheres or Morse functions in Reeb's theorem are simplest examples. Special generic maps and the manifolds of their domains have been studied well. Their structures are simple and this help us to study explicitly. As important properties, they have been shown to restrict the topologies and the differentiable structures of the manifolds strongly by Saeki and Sakuma, followed by Nishioka, Wrazidlo and the author. To cover wider classes of manifolds as the domains, the author previously introduced a class generalizing the class of special generic maps and smaller than our class: {\it simply generalized special generic maps}.


		
	\end{abstract}
	
	
	\maketitle
	\section{Introduction.}
	\label{sec:1}
	{\it Special generic} maps are smooth maps which are locally projections or the product maps of Morse functions and the identity maps on disks. Morse functions with exactly two singular points on spheres or Morse functions in Reeb's theorem are simplest examples. Canonical projections of so-called {\it unit spheres} are also simplest examples. 
	
    Pioneering studies are \cite{burletderham, calabi, furuyaporto} for example. Since the 1990s, their algebraic topological and differential topological properties have been studied by Saeki and Sakuma (\cite{saeki1,saeki2,saekisakuma1,saekisakuma2}), followed by Nishioka (\cite{nishioka}), Wrazidlo (\cite{wrazidlo1,wrazidlo2}) and the author (\cite{kitazawa1,kitazawa2,kitazawa3,kitazawa4,kitazawa5,kitazawa6,kitazawa7,kitazawa8,kitazawa11}).

    Some elementary manifolds such as ones represented as connected sums of the products of two spheres admit natural special generic maps in considerable cases. On the contrary, the differentiable structures of spheres and some elementary manifolds admitting special generic maps are restricted strongly. Homology groups of the manifolds are also restricted in considerable cases. The author has started studies on the cohomology rings of the manifolds.
    These studies are due to the fact that special generic maps have simple and nice structures.
    
    We introduce some fundamental notions, terminologies and notation. ${\mathbb{R}}^k$ is the $k$-dimensional Euclidean space, which is a simplest $k$-dimensional smooth manifold for an arbitrary positive integer $k$. It is, by considering a standard Euclidean metric, a Riemannian manifold. $||x|| \geq 0$ denotes th distance between $x \in {\mathbb{R}}^k$ and the origin $0$. $\mathbb{R}:={\mathbb{R}}^1$ and $\mathbb{Z}$ denotes the ring of all integers.
    
    $S^k:=\{x \in {\mathbb{R}}^{k+1} \mid ||x||=1\}$ denotes the $k$-dimensional unit sphere. It is a $k$-dimensional smooth compact and connected submanifold if $k \geq 2$ and the two-point set with the discrete topology if $k=1$.  $D^k:=\{x \in {\mathbb{R}}^{k} \mid ||x|| \leq 1\}$ denotes the $k$-dimensional unit disk for an arbitrary integer $k \geq 1$. It is a $k$-dimensional smooth compact and connected submanifold.
    
    For a non-empty topological space $X$ having the structure of a cell complex, we can define the dimension uniquely, denoted by $\dim X$. A topological manifold is known to have the structure of a CW complex. A smooth manifold is known to have the structure of a polyhedron and more precisely, in some canonical way we can define the unique PL structure. This is seen as a {\it PL manifold} canonically. It is well-known that topological manifolds whose dimensions are at most $3$ have the unique structures of polyhedra and that topological spaces homeomorphic to polyhedra whose dimensions are at most $2$ have the unique structures of polyhedra. Theory related to such uniqueness is discussed in \cite{moise} for example. 
    
    For a differentiable map $c:X \rightarrow Y$, a point $x \in X$ is a {\it singular} point if the rank of the differential ${dc}_x$ there is smaller than the minimum between $\dim X$ and $\dim Y$. The singular set of $c$ is the set of all singular points of $c$ and let $S(c)$ denote the set.
   
    We define and discuss special generic maps later. However, we explain about these maps shortly. A special generic map is a smooth map from an $m$-dimensional manifold with no boundary into $n$-dimensional one with $m \geq n$. If the manifold of the domain is closed, then the image is regarded as an $n$-dimensional smoothly immersed compact manifold. The preimage of each point in the interior of the immersed manifold is diffeomorphic to $S^{m-n}$. This gives a projection over the interior of the manifold. The preimage of each point in the boundary is a point. Around each point in the boundary, it is the product map of a Morse function with exactly one singular point on a disk of dimension $m-n+1$ and the identity map on a disk of dimension $n-1$. Propositions \ref{prop:1} and \ref{prop:2}, presented in the next section, are on this.

    \begin{Prob}
    	\label{prob:1}
    	 Can you formulate new nice classes of smooth maps having simple and nice structures and properties special generic maps have?
    	 \end{Prob}
     \begin{Prob}
     	\label{prob:2}
    	Can the classes of Problem \ref{prob:1} cover wider classes of manifolds to study as the domains of these maps?
    \end{Prob}
Essentially same problems are also asked in the previous preprint of the author \cite{kitazawa9}. There the author introduced {\it simply generalized special generic maps}. In our paper, we generalize this class as the class of {\it special-generic-like} maps or {\it SGL} maps.

We explain about simply generalized special generic maps shortly. The preimage of a point in the interior of the $n$-dimensional smoothly immersed manifold is replaced by the product of manifolds diffeomorphic to unit spheres. Around each point in the boundary, a Morse function on a disk is replaced by the composition of a projection onto the disk with a Morse function with exactly one singular point there. Each preimage of the projection is the product of manifolds diffeomorphic to unit spheres. The Morse functions are replaced by {\it Morse-Bott} functions. For {\it Morse-Bott} functions, see a related pioneering study \cite{bott} for example.

Note that this class respects local structures of so-called {\it moment maps} on so-called {\rm (}symplectic{\rm )} {\it toric} manifolds. The class of moment maps and that of symplectic toric manifolds are important in symplectic geometry and toric geometry. In most cases, special generic maps are not admitted by such manifolds. This has been conjectured by us and the author has shown results seeming to be related to this. \cite{kitazawa4,kitazawa6} show the non-existence of special generic maps on complex projective spaces except for the $1$-dimensional case or $2$-dimensional spheres. They are simplest (symplectic) toric manifolds. \cite{delzant} is a pioneering study on symlectic toric manifolds. See also \cite{buchstaberpanov} for example. Exposition related to this history is also in \cite{kitazawa9}
    
In our new class, the preimage of each point in the interior of the immersed manifold and Morse(-Bott) functions used around the boundary are generalized.
 
We present our main results as Main Theorems. Hereafter, elementary algebraic topology, more precisely, elementary (co)homology theory is important. We do not explain about this rigorously or systematically and we also expect we have related knowledge. For related studies, see \cite{hatcher} and we also abuse notions, terminologies and notation here or ones which seem to be generally used or familiar to us.
\begin{MainThm}[Theorem \ref{thm:2}]
	\label{mthm:1}
	Let $k>0$ be a positive integer.
	Let $m>n$ be an arbitrary positive integers satisfying $m-n \geq k$. Let $f:M \rightarrow N$ be an SGL map from an $m$-dimensional closed and connected manifold into an $n$-dimensional connected manifold $N$ with no boundary which is represented as the composition of a smooth surjection $q_f:M \rightarrow W_f$ onto a compact and $k$-connected connected manifold $W_f$ with a smooth immersion $\bar{f}:W_f \rightarrow N$. Suppose that
	a family $\{p_j\} \subset {\rm Int}\ W_f$ of finitely many points satisfying the following conditions exists .
	 \begin{enumerate}
	 	\item The preimage ${q_f}^{-1}(p_j)$ for $p_j \in {\rm Int}\ W_f$ is diffeomorphic to a closed and {\rm (}$k-1${\rm )}-connected manifold $F$.
	 	\item For a smooth curve ${\alpha}_{p_j}:[0,1] \rightarrow W_f$ satisfying ${\alpha}_{p_j}(0)=p_j$, ${\alpha}_{p_j}((0,1))
	\subset {\rm Int}\ W_f$ and ${\alpha}_{p_j}(1) \in \partial W_f$ and a so-called "transversality", presented later in several situations, the inclusion ${{\alpha}_{p_j}}^{-1}(0) \subset {{\alpha}_{p_j}}^{-1}([0,1])$ gives the kernel ${\rm Ker} \ {{\alpha}_{p_j}}_{\ast}$ of the naturally defined homomorphism between the $k$-th homotopy groups ${\pi}_k({{\alpha}_{p_j}}^{-1}(0))$ and ${\pi}_k({{\alpha}_{p_j}}^{-1}([0,1]))$.
	\item The union ${\bigcup}_j {\rm Ker}\ {{\alpha}_{p_j}}_{\ast}$ generates ${\pi}_k(F)$.
	\end{enumerate} 
Then $M$ is $k$-connected.
	\end{MainThm}
\begin{MainThm}[Theorem \ref{thm:3}]
		\label{mthm:2}
	Let $n$ be an arbitrary positive integer. Suppose that an $n$-dimensional connected manifold $N$ with no boundary and a smooth immersion $\bar{f}:\bar{N} \rightarrow N$ of an n-dimensional compact and simply-connected manifold $\bar{N}$ which has at least two boundary components are given. Let $F$ be a closed, connected and orientable surface which is not a sphere.

Then we have an SGL map $f:M \rightarrow N$ from some $m$-dimensional closed and simply-connected manifold $M$ into $N$ for the case $k=1$ of Main Theorem \ref{mthm:1} with the notation being abused and $W_f$ and $\bar{N}$ being identified suitably as smooth manifolds.
\end{MainThm}

We explain about the content of our paper. The second section is for preliminaries. There we also review special generic maps for example as fundamental objects. We define our new class rigorously in the third section. The fourth section presents our proof of Main Theorems with some examples. Main Theorems \ref{mthm:1} and \ref{mthm:2} are presented again in revised forms. Main Theorems \ref{mthm:3}, \ref{mthm:4} and \ref{mthm:5} are also presented as another new result. We close our paper by Main Theorem \ref{mthm:6} as a meaningful remark on singularities of the maps in Main Theorems \ref{mthm:2}, \ref{mthm:3}, \ref{mthm:4} and \ref{mthm:5}.	\ \\
\ \\
	{\bf Conflict of interest.} \\
	The author is a member of the project JSPS KAKENHI Grant Number JP22K18267 "Visualizing twists in data through monodromy" (Principal Investigator: Osamu Saeki). Our present study is supported by the project. \\
	\ \\
	{\bf Data availability.} \\
	Data essentially supporting our present study are all contained in our present paper.
	
	\section{Preliminaries.}
\subsection{Diffeomorphisms.}
	A {\it diffeomorphism} between smooth manifolds means a smooth map which has no singular points and which is a homeomorphism. A {\it diffeomorphism on a manifold} is a diffeomorphism from the (smooth) manifold onto itself. Two manifolds are {\it diffeomorphic} if and only if there exists a diffeomorphism between these manifolds. This naturally gives an equivalence relation on the family of all smooth manifolds with their corners being eliminated in a well-known canonical way. 
	These operations always give mutually diffeomorphic manifolds with no corners for a fixed manifold.
	We can define {\it PL homeomorphic manifolds} using PL homeomorphisms or piecewise smooth homeomorphisms similarly.
	
	The {\it diffeomorphism group} of a manifold is the space consisting of all diffeomorphisms on the manifold where the so-called {\it Whitney $C^{\infty}$ topology} is given as its topology. It is also a topological group and a so-called {\it infinite dimensional Lie groups}. 
	The {\it Whitney $C^{\infty}$ topologies} on spaces of smooth maps between smooth manifolds are natural and important topologies. Such spaces are fundamental and important spaces in the singularity theory of differentiable maps and (applications to) differential topology of manifolds. For this see \cite{golubitskyguillemin}.
\subsection{Smooth bundles and linear bundles.}		
	A {\it smooth} bundle means a bundle whose fiber is a smooth manifold and whose structure group is regarded as some subgroup of the diffeomorphism group.
	{\it Linear} bundles form an important subclass. A {\it linear} bundle is a bundle whose fiber is a Euclidean space, a unit disk, or a unit sphere, and whose structure group consists of linear transformations. Note that we can define linear transformations here naturally.
	
	To know general theory of bundles systematically, see \cite{steenrod} for example. For linear bundles and so-called characteristic classes of them, see \cite{milnorstasheff} for example. 
	
	\subsection{Special generic maps.}
	\begin{Def}
		\label{def:1}
		A smooth map $c:X \rightarrow Y$ between two smooth manifolds with no boundaries is {\it special generic} if at some small neighborhood of each singular point $p \in X$, we can choose suitable local coordinates around $p$ and $c(p)$ and $c$ can be represented by $(x_1,\cdots,x_{\dim X}) \rightarrow (x_1,\cdots,x_{\dim Y-1},{\Sigma}_{j=1}^{\dim X-\dim Y+1} {x_{\dim Y+j-1}}^2)$ locally for the local coordinates.
	\end{Def}
	 A canonical projection of a unit sphere, mapping $(x_1,x_2) \in S^k \subset {\mathbb{R}}^{k+1}={\mathbb{R}}^{k_1} \times {\mathbb{R}}^{k_2}$ to $x_1 \in {\mathbb{R}}^{k_1}$, is a special generic map where $k \geq 1$, $k_1,k_2 \geq 1$ and $k=k_1+k_2$. To check that this is special generic maps is a kind of elementary exercises on smooth manifolds, smooth maps, Morse functions and differentiable maps. 

	\begin{Prop}[\cite{saeki1,saeki2}]
		\label{prop:1}
		Let $m \geq n \geq 1$ be integers. Given a special generic map $f:M \rightarrow N$ on an $m$-dimensional closed and connected manifold $M$ into an $n$-dimensional connected manifold $N$ with no boundary. This enjoys the following properties.
		\begin{enumerate}
			\item \label{prop:1.1}
			A suitable $n$-dimensional compact and connected smooth manifold $W_f$, a suitable smooth surjection $q_f:M \rightarrow W_f$ and a suitable smooth immersion $\bar{f}:W_f \rightarrow N$ exist and we have a relation $f=\bar{f} \circ q_f$. Furthermore, $q_f$ can be chosen as a smooth map whose restriction to the singular set $S(f)$ of $f$ is a diffeomorphism onto the boundary $\partial W_f \subset W_f$ where the manifold of the target is restricted.
			\item \label{prop:1.2}
			We have some small collar neighborhood $N(\partial W_f)$ of the boundary $\partial W_f \subset W_f$ with the following two properties.
			\begin{enumerate}
				\item \label{prop:1.2.1}
				The composition of the restriction of $q_f$ to the preimage ${q_f}^{-1}(N(\partial W_f))$ with the canonical projection to $\partial W_f$ is the projection of some linear bundle whose fiber is the unit disk $D^{m-n+1}$.
				\item \label{prop:1.2.2}
Suppose that $\partial W_f$ is not empty. The restriction of $q_f$ to the preimage of $W_f-{\rm Int}\ N(\partial W_f)$ is the projection of some smooth bundle whose fiber is the unit sphere $S^{m-n}$. In some specific cases, the bundle is regarded as a linear one and the case $m-n=0,1,2,3$ satisfies the conditions.
		\end{enumerate}
	\end{enumerate}
	\end{Prop}

\begin{Def}[E. g. \cite{kitazawa8}]
	\label{def:2}
	In Proposition \ref{prop:1}, we call the bundle in (\ref{prop:1.2.1}) the {\it boundary linear bundle} {\rm (}of $f${\rm )}. We call the bundle of (\ref{prop:1.2.2}) the {\it internal smooth bundle} {\rm (}of it{\rm )}.
\end{Def}
The following gives simplest special generic maps.

	\begin{Prop}[\cite{saeki1}]
		\label{prop:2}
Let $m \geq n \geq 1$ be integers. Let $\bar{N}$ be an $n$-dimensional smooth, compact and connected manifold whose boundary is not empty and $N$ an $n$-dimensional smooth connected manifold with no boundary. Assume also that a smooth immersion ${\bar{f}}_N:\bar{N} \rightarrow N$ is given.
			
		Then we have a suitable $m$-dimensional closed and connected manifold $M$ some special generic map $f:M \rightarrow N$ and have the following two.
		\begin{enumerate}
			\item The property {\rm (}\ref{prop:1.1}{\rm )} of Proposition \ref{prop:1} is enjoyed where $W_f$ and $\bar{N}$ are identified as smooth manifolds in a suitable way with the relation ${\bar{f}}_N=\bar{f}$.
			\item A boundary linear bundle and an internal smooth bundle of $f$ are trivial.
		\end{enumerate}
 
	\end{Prop}

This is also regarded as an elementary exercise on smooth maps and differential topology of manifolds. Remark \ref{rem:1} with Example \ref{ex:2} gives related exposition.  

	Canonical projections of unit spheres are simplest special generic maps. We present another simplest example.
	\begin{Ex}
		\label{ex:1}
		Let $l$ be an arbitrary positive integer. Let $m$ and $n$ be integers satisfying the condition $m \geq n \geq 2$. Assume that an integer $1 \leq n_j \leq n-1$ is defined for each integer $1 \leq j \leq l$. We take a connected sum of $l>0$ manifolds in the smooth category where the $j$-th manifold is $S^{n_j} \times S^{m-n_j}$. Thus we have a smooth manifold $M_0$. We have a special generic map $f_0:M_0 \rightarrow {\mathbb{R}}^n$ as in Proposition \ref{prop:2}. More precisely, we have $f_0$ in such a way that the image is represented as a boundary connected sum of $l$ manifolds taken in the smooth category with the $j$-th manifold diffeomorphic to $S^{n_j} \times D^{n-n_j}$ 
		
	\end{Ex}
	
    Hereafter, a {\it homotopy sphere} is a smooth manifold homeomorphic to a (unit) sphere whose dimension is positive. 
	A {\it standard} sphere is a homotopy sphere diffeomorphic to some unit sphere. An {\it exotic sphere} is a homotopy sphere which is not diffeomorphic to any unit sphere. It is well-known that $4$-dimensional exotic spheres are still undiscovered. Except these $4$-dimensional cases, homotopy spheres are known to be PL homeomorphic to standard spheres where they are seen as the PL manifolds defined canonically. In this philosophy, $4$-dimensional exotic spheres are known to be not PL homeomorphic to standard spheres.
	
	As a kind of appendices, we present known results on special generic maps and manifolds admitting them in several situations.  
	
	\begin{Thm}[\cite{saeki1,saeki2}]
		\label{thm:1}
		\begin{enumerate}
			\item
			\label{thm:1.1}
			Let $m$ be an arbitrary integer satisfying $m \geq 2$.
			An $m$-dimensional closed and connected manifold $M$ admits a special generic map $f:M \rightarrow {\mathbb{R}}^2$ if and only if $M$ is either of the following two.
			\begin{enumerate}
			\item A homotopy sphere which is not a $4$-dimensional exotic sphere.
				\item A manifold
				represented as a connected sum of smooth manifolds each of which is the total space of some smooth bundle over $S^1$. Furthermore, the connected sum is taken in the smooth category and the fiber of each bundle here is a homotopy sphere which is not a $4$-dimensional exotic sphere.
				\end{enumerate}
		\item
		\label{thm:1.2}
			Let $m$ be an arbitrary integer greater than or equal to $4$. Let $M$ be an $m$-dimensional closed and simply-connected manifold $M$. If a special generic map $f:M \rightarrow {\mathbb{R}}^3$ exists, then $M$ is either of the following two.
			\begin{enumerate}
			\item A homotopy sphere which is not a $4$-dimensional exotic sphere.
				\item A manifold
				represented as a connected sum of smooth manifolds each of which is the total space of a smooth bundle over $S^2$. Furthermore, the connected sum is taken in the smooth category and the fiber of each bundle here is a homotopy sphere which is not a $4$-dimensional exotic sphere.
			\end{enumerate}
			In the case $m=4,5,6$, the converse is also true. In such a case, a fiber of each bundle is an {\rm (}$m-2${\rm)}-dimensional standard sphere and the total spaces of the bundles are replaced by the total spaces of linear bundles without changing the fibers and the base spaces.
			\item
			\label{thm:1.3}
		Both in the cases {\rm (}\ref{thm:1.1}{\rm )} and {\rm (}\ref{thm:1.2}{\rm )},
		consider the manifold $M$ which is not a homotopy sphere. $M$ admits a special generic map as in Example \ref{ex:1} such that an internal smooth bundle and a boundary linear bundle of it may not be trivial. Let $n_j=1$ and $n_j=2$ in Example \ref{ex:1}, respectively.
		
	\end{enumerate}
	\end{Thm}
In addition, \cite{nishioka} and \cite{kitazawa5} have solved variants of problems of Theorem \ref{thm:1} (\ref{thm:1.2}) and obtained answers in the cases $(m,n)=(5,4), (6,4)$ where $n$ is the dimension of the Euclidean space of the target, respectively. 

\subsection{Some exposition on elementary algebraic topology and differential topology.}

We omit systematic and rigorous exposition on elementary algebraic topology as we have said in the first section. However, we need exposition for some. For systematic studies, consult \cite{hatcher} again for example.

One of such exposition is on {\it fundamental classes} of connected and compact (oriented) manifolds.

Let $A$ be a commutative ring having a unique identity element $1_A$ which is different from the zero element $0_A$. $1_A$ and $-1_A$ are generators of $A$. $A$ is also seen as a module over $A$ canonically.
For any compact and connected oriented manifold $X$, $H_{\dim X}(X, \partial X; A)$ is isomorphic to $A$ as the module. A generator is given according to the orientation of $X$. Note that for example, orientations are not needed in the case $A:=\mathbb{Z}/2\mathbb{Z}$ or the commutative ring of order $2$.

For a manifold $Y$, consider an embedding $i_X:X \rightarrow Y$ satisfying suitable conditions according to the category where we argue. For example, in the smooth category, the embedding is smooth and in the PL or equivalently, in the piecewise smooth category, this is defined to be piecewise smooth. Furthermore, $X$ is embedded {\it properly}. In other words, the boundary is embedded into the boundary and the interior is embedded into the interior and (in the smooth category) $X$ must be embedded in a so-called {\it generic} way. More precisely, we need "transversality".  
If $a \in H_j(Y, \partial Y; A)$ is the value of the homomorphism ${i_{X}}_{\ast}:H_{\dim X}(X, \partial X; A) \rightarrow H_{\dim X}(Y, \partial Y; A)$ induced canonically from the embedding at the fundamental class $[X] \in H_{\dim X}(X, \partial X; A)$, then $a$ is {\it represented} by the submanifold $i_X(X)$.

We add exposition on "transversality" related to our smooth embedding $i_X:X \rightarrow Y$. We generally consider a smooth embedding satisfying a nice condition on the dimensions of subspaces of tangent vector spaces and the images of differentials. More rigorously, the dimension of the intersection of the image of the differential $d {i_X}_p$ of the embedding $i_X:X \rightarrow Y$ at each point $p \in \partial X$ and the tangent space at $i_X(p) \in \partial Y$ must be calculated as $\dim X+\dim \partial Y-\dim Y=\dim X-1$. This is also fundamental and important in the singularity theory of differentiable maps and applications to differential topology of manifolds. Consult \cite{golubitskyguillemin} again for systematic studies for example.

For compact, connected and oriented manifolds, so-called {\it Poincar\'e duals} to elements of the (co)homology groups and Poincar\'e duality (theorem) are important.

We explain about duals in modules and {\it cohomology duals}. Let $B_A$ be a module over $A$ having a unique maximal free submodule and let the rank of this submodule be finite and $l$. Suppose that we have a basis ${\mathcal{B}}_A:=\{e_j\}_{j=1}^l$ of it consisting of elements which are not divisible by elements which are not units of $A$. We can define a homomorphism ${e_j}^{\ast,{\mathcal{B}}_A}$ from $B_A$ into $A$ uniquely by the relation: ${e_{j_1}}^{\ast,{\mathcal{B}}_A}(e_{j_2}):=1_A$ in the case $j_1=j_2$ ${e_{j_1}}^{\ast,{\mathcal{B}}_A}(e_{j_2}):=0_A$ in the case $j_1 \neq j_2$. This is the {\it dual} to $e_j$ {\it respecting the basis} ${\mathcal{B}}_A:=\{e_j\}_{j=1}^l$. If $B_A$ is a homology group of some topological space, then, the element is the {\it cohomology dual} respecting the basis. 

\section{Our new class of smooth maps like special generic maps.}

\begin{Def}
	\label{def:3}
	Let $m \geq n \geq 1$ be integers.
	A {\it special-generic-like} map or an {\it SGL} map is a smooth map $f:M \rightarrow N$ on an $m$-dimensional closed and connected manifold $M$ into an $n$-dimensional smooth manifold $N$ with no boundary enjoying the following properties.
	\begin{enumerate}
	 \item \label{def:3.1} The image $f(M)$ is the image of some smooth immersion ${\bar{f}}_N:\bar{N} \rightarrow N$ of some $n$-dimensional compact and connected manifold $\bar{N}$.
	 
	 	\item \label{def:3.2}
 As in Proposition \ref{prop:1}, we have some smooth surjection $q_f:M \rightarrow W_f$ with the manifold $\bar{N}$ being identified in a suitable way with $W_f$ as a smooth manifold and have the relation $f={\bar{f}}_N \circ q_f$. Furthermore, we can choose $q_f$ as a map whose restriction to the singular set $S(f)$ gives a diffeomorphism onto the boundary $\partial W_f$.
	 	\item \label{def:3.3}
 We have some small collar neighborhood $N(\partial W_f)$ of the boundary $\partial W_f\subset W_f$ and the composition of the restriction of $q_f$ to the preimage with the canonical projection to the boundary is the projection of some smooth bundle over $\partial W_f$. This is also as in Proposition \ref{prop:1}.
\item \label{def:3.4}
 On the collar neighborhood $N(\partial W_f)$ and the preimage ${q_f}^{-1}(N(\partial W_f))$, it is represented as the product map of the following two smooth maps for suitable local coordinates around each point $p$ of $\partial W_f \subset N(\partial W_f)$. 
\begin{enumerate}
\item \label{def:3.4.1}
 A smooth function ${\tilde{f}}_{p,\partial W_f}$ on an ($m-n+1$)-dimensional smooth compact and connected manifold $E_p$.  
\begin{enumerate}
\item
\label{def:3.4.1.1}
 The image of the function ${\tilde{f}}_{p,\partial W_f}$ can be denoted by $[{\rm min}_p,{\rm max}_p]$. We have ${{\tilde{f}}_{p,\partial W_f}}^{-1}({\rm min}_p)=\partial E_p$ or ${{\tilde{f}}_{p,\partial W_f}}^{-1}({\rm max}_p)=\partial E_p$.
 The singular set of the function ${\tilde{f}}_{\rm p,\partial W_f}$ is in the interior ${\rm Int}\ E_p$.
 \item \label{def:3.4.1.2} For values of the function, the preimage of a value contains some singular points if and only if it is the maximum or the minimum, which is defined uniquely. If ${{\tilde{f}}_{p,\partial W_f}}^{-1}({\rm min}_p)=\partial E_p$, then ${{\tilde{f}}_{p,\partial W_f}}^{-1}({\rm max}_p)$ is a subpolyhedron of dimension at most $m-n$ where $E_p$, the outer manifold $M$ and related smooth manifolds are regarded as the canonically defined PL manifolds. If ${{\tilde{f}}_{p,\partial W_f}}^{-1}({\rm max}_p)=\partial E_p$, then ${{\tilde{f}}_{p,\partial W_f}}^{-1}({\rm min}_p)$ is a subpolyhedron of dimension at most $m-n$ where $E_p$, the outer manifold $M$ and related smooth manifolds are regarded as the canonically defined PL manifolds.
 \end{enumerate}
 \item \label{def:3.4.2} The identity map on a small open neighborhood of $p$ where the neighborhood is considered in $\partial W_f$
\end{enumerate}
 
\item
\label{def:3.5}
 The restriction of $q_f$ to the preimage ${q_f}^{-1}(W_f-{\rm Int}\ N(\partial W_f))$ is the projection of some smooth bundle whose fiber is some smooth closed manifold $F$. The boundary of the manifold $E_p$ of the domain of the function used for the product maps before is diffeomorphic to $F$.

	 \end{enumerate}
	
\end{Def}
We discuss Definition \ref{def:3} further by assuming $m>n$ and the manifold $F$ and the preimage ${q_f}^{-1}(p^{\prime})$ of any point in $p^{\prime} \in \partial W_f$ to be connected. Fix the manifold $F$.

We fix an arbitrary point $q$ in the interior of $W_f$. 
In Definition \ref{def:3}, for each connected component $C$ of the boundary $\partial W_f$, we can define a smooth embedding ${\alpha}:[0,1] \rightarrow W_f$ enjoying the following properties.
\begin{itemize}
	\item ${\alpha}(0)=q$
	\item ${\alpha}([0,1)) \subset {\rm Int}\ W_f$
	\item ${\alpha}(1) \in \partial W_f$
	\item The intersection of the image of the differential of $\alpha$ at $1$ and the tangent vector space of $\partial W_f$ at $p:=\alpha(1)$ is of dimension $\dim \partial W_f-1$. It is a condition on "transversality".
\end{itemize}
The preimage ${q_f}^{-1}(\alpha([0,1]))$ is regarded as the manifold $E_p$ of the domain of a smooth function ${\tilde{f}}_{p,\partial W_f}$. 
More rigorously, it is regarded as a function {\it $C^{\infty}$ equivalent} to the function ${\tilde{f}}_{p,\partial W_f}$ and see \cite{golubitskyguillemin} for such a notion.
We have the inclusion of $F$ into $E_p$. Let $A$ be a commutative ring (having the unique identity element which is not the zero element). 
 We can define the homomorphism from $H_j(F;A)$ into $H_j(E_p;A)$ and the homomorphism ${\pi}_j(F)$ into ${\pi}_j(E_p)$ induced by the inclusion uniquely.
\begin{Def}
	\label{def:4}
We call a homomorphism between the groups above a {\it reference homomorphism along the curve} $\alpha$.
 \end{Def}

In considerable situations, such homomorphisms are defined uniquely (modulo suitable equivalence relation). However, we do not consider such arguments precisely. 

\begin{Ex}
\label{ex:2}
	
	We present examples by abusing terminologies and notation in Definitions \ref{def:3} and \ref{def:4}. 
	\begin{enumerate}
		\item \label{ex:2.1}
 Special generic maps (whose singular sets are not empty on closed and connected manifolds) are seen as specific SGL maps. We explain about this. We explain about a {\it height function of a unit sphere}. This gives a specific case of canonical projections of unit spheres and a Morse function. A {\it height function of a unit disk} is defined as the restriction of the height function of the unit sphere to the preimage of the half-line, the intersection of the line $\mathbb{R}$ and $\{t \geq 0 \mid t \in \mathbb{R}\}$.
$F$ is taken as the unit sphere $S^{m-n}$, $E_p$ is taken as the unit disk $D^{m-n+1}$, and the function ${\tilde{f}}_{p,\partial W_f}$ is a height function here. Reference homomorphisms along these curves are always the zero homomorphisms except the case of degree $0$ in Definition \ref{def:4}. 
		\item \label{ex:2.2}
Simply generalized special generic maps (whose singular sets are not empty on closed and connected manifolds) give cases where $F$ is the product of standard spheres in Definitions \ref{def:3} and \ref{def:4}. We explain about local functions used for the product maps on the boundary. $E_p$ is defined as
the product of a copy of some unit disk $D^{k_p+1}$ and finitely many standard spheres where $k_p \geq 1$ is an integer. Here the family of the spheres here are obtained by removing exactly one sphere in the family of the spheres for the product $F$. The function ${\tilde{f}}_{p,\partial W_f}$ is represented as the projection of a trivial bundle over the copy of $D^{k_p+1}$ with a height function.
The fiber is diffeomorphic to the product of standard spheres where the family of the spheres here are obtained by removing the exactly one sphere in the family of the spheres for the product $F$ as before. 
The function is a Morse-Bott function.
Let $A$ be a principal ideal domain having the identity element different from the zero element.
In simplest cases, our reference homomorphisms along our curves are homomorphisms enjoying the following properties where the coefficient ring is $A$.
\begin{itemize}
\item The kernels are free and of rank $1$. 
\item The kernel is generated by an element represented by a subspace of the product of standard spheres where the product of the spheres and the subspace are presented in the following.
\begin{itemize}
\item The product of the spheres is canonically identified with $F$, diffeomorphic to and identified with the product ${\prod}_{j=1}^{l+1} S^{k_j}$, and the dimension is $m-n$ where $l \geq 0$ is an integer: if $l=0$, then the Morse-Bott function is a height function on the copy of $D^{k_p+1}$ with $k_p=k_1$. This does not depend on which point we choose in $\partial W_f$. 
\item The subspace is represented as ${\prod}_{j=1}^{l} S_p(S^{k_j}) \subset {\prod}_{j=1}^{l} S^{k_j}$ satisfying the following conditions: $S_p(S^{k_j})=S^{k_j}$ for exactly one integer $j=j_p$ and suitable one-point set $\{{\ast}_{j,p}\}$ for remaining integer $j \neq j_p$. This does not depend on which point we choose in the connected component $C \ni p$ of $\partial W_f$. 

\end{itemize} 
\end{itemize} 
Note that in \cite{kitazawa9}, simply generalized special generic maps in simplest cases here and the homology groups and the cohomology rings of the manifolds are studied.

Of course, special generic maps are also simply generalized special generic maps.

			\end{enumerate}
\end{Ex}
We give an additional remark.
\begin{Rem}
\label{rem:1}
If the following are given, then we have a natural SGL map with some simple structure easily by using the canonically obtained trivial bundles over the complementary set of a suitable small collar neighborhood of $\partial \bar{N} \subset \bar{N}$ and the product maps around the collar neighborhood. Proposition \ref{prop:2} gives one of simplest specific cases.
\begin{itemize}
\item Integers $m \geq n \geq 1$.
\item A smooth immersion ${\bar{f}}_N:\bar{N} \rightarrow N$ of some $n$-dimensional compact and connected manifold $\bar{N}$ into an $n$-dimensional smooth manifold $N$ with no boundary.
\item A suitable smooth function ${\tilde{f}}_{p,\partial W_f}$ on an ($m-n+1$)-dimensional smooth compact and connected manifold $E_p$.
\end{itemize}
We first construct these bundles, projections and these product maps and gluing them suitably via isomorphisms of the trivial smooth bundles defined canonically on the boundaries. Simplest examples are given by considering the product maps of the identity maps between the base spaces and the fibers. Note here that for the identity maps here, identifications are given suitably and naturally beforehand.

This is important in construction and used in the proof of Main Theorem \ref{mthm:2}.
Related to Example \ref{ex:2} (\ref{ex:2.2}), one of our main work of \cite{kitazawa9} is regarded as a work constructing simply generalized special generic maps with prescribed reference homomorphisms along given curves and studying the homology groups and the cohomology rings of the obtained manifolds of the domains. Of course reference homomorphisms along curves are not introduced there. Note also that this generalizes our previous work on the homology groups and the cohomology rings of the manifolds in the cases of special generic maps \cite{kitazawa11}. More precisely, \cite{kitazawa11} studies the case where the isomorphisms of the trivial smooth bundles defined canonically on the boundaries of the canonically obtained trivial bundles are regarded as the product maps of the identity maps between the base spaces and the fibers.
\end{Rem}

\section{On Main Theorems.}

We present Main Theorems \ref{mthm:1} and \ref{mthm:2} again in terms of Definitions \ref{def:3} and \ref{def:4}.

\begin{Thm}[Main Theorem \ref{mthm:1}]
	\label{thm:2}
	Let $m>n$ be an arbitrary positive integer satisfying $m-n>0$. Let $f:M \rightarrow N$ be an SGL map from an $m$-dimensional closed and connected manifold into an $n$-dimensional connected manifold $N$ with no boundary which is represented as the composition of a smooth surjection $q_f:M \rightarrow W_f$ onto an $n$-dimensional compact and $k$-connected connected manifold $W_f$ with a smooth immersion $\bar{f}:W_f \rightarrow N$.

We also assume the following conditions.
\begin{itemize}
	\item $k<m-n$.
	\item The preimage ${q_f}^{-1}(p)$ of a point $p \in {\rm Int}\ W_f$ is diffeomorphic to a closed and {\rm (}$k-1${\rm )}-connected manifold $F$. 
	\item For any point $p \in \partial W_f$, ${q_f}^{-1}(p)$ is connected and at most $k^{\prime}$-dimensional with $k^{\prime} \leq m-n$.
	\item $k+k^{\prime}+n-1<m$.
    \item There exist finitely many smooth curves in ${\rm Int}\ W_f$ and that the reference homomorphisms along these curves between the $k$-th homotopy groups are defined. Furthermore, the union of their kernels, uniquely defined, and ${\pi}_k(F)$ coincide. 
\end{itemize}
Then $M$ is $k$-connected.
\end{Thm}
\begin{proof}
	We can take a smooth submanifold ${S^k}_0$ diffeomorphic to $S^k$ in $M$.
	${q_f}^{-1}(\partial W_f)$ is a polyhedron of dimension at most $n-1+k^{\prime}$.  
	 $k+(n-1+k^{\prime})<m$ implies that we can smoothly isotope the submanifold ${S^k}_0$ apart from ${q_f}^{-1}(\partial W_f)$.
	 
	 $S^k$ is, thus, homotopically, regarded as in the total space of the trivial smooth bundle over a smoothly embedded copy of the unit disk $D^{n}$ in ${\rm Int}\ W_f$ defined by considering the preimage by $q_f$. Its fiber is ($k-1$)-connected and diffeomorphic to $F$. By the assumption on our reference homomorphisms, $S^k$ is shown to be null-homotopic. This completes the proof.
	
\end{proof}

\begin{Ex}
\label{ex:3}
\begin{enumerate}
\item 
\label{ex:3.1}
According to \cite{saeki1}, for a special generic map into the Euclidean space or a connected manifold which is not compact and has no boundary, $q_f$ induces isomorphisms between the homology groups and the homotopy groups whose degrees are at most $m-n$. Theorem \ref{thm:2} generalizes this.
\item 
\label{ex:3.2}
 Consider a map represented as the composition of the projection of a trivial smooth bundle $S^{k_1} \times S^{k_2}$ over $S^{k_1}$ whose fiber is diffeomorphic to $S^{k_2}$ with a canonical projection into ${\mathbb{R}}^{k_1-1}$ where the conditions $k_1 \geq 2$ and $k_2 \geq 2$ are assumed. This is for Theorem \ref{thm:2} where $m=k_1+k_2$, $n=k_1-1$, $N:={\mathbb{R}}^n$, $k=1$, $W_f:=D^{n}$, $\bar{f}:W_f \rightarrow N$ is the canonical smooth embedding, and $F=S^1 \times S^{k_2}$.
\end{enumerate}
\end{Ex}
\begin{Def}
	An SGL map in Theorem \ref{thm:2} is said to be a {\it $k$-connected-SGL} map.
\end{Def}

\begin{Thm}[Main Theorem \ref{mthm:2}]
\label{thm:3}
	Let $n$ be an arbitrary positive integer. Suppose that an $n$-dimensional connected manifold $N$ with no boundary and a smooth immersion $\bar{f}:\bar{N} \rightarrow N$ of an n-dimensional compact and simply-connected manifold $\bar{N}$ having at least two boundary components are given.
Let $m=n+2$. Let $F$ be a closed, connected and orientable surface which is not a sphere. Then we have a 1-connected-SGL map $f:M \rightarrow N$ on some $m$-dimensional closed and simply-connected manifold $M$ into $N$ satisfying the conditions of Theorem \ref{thm:2} with the notation being abused and $\bar{N}$ and $W_f$ being identified suitably as smooth manifolds.
\end{Thm}

\begin{proof}
	We review our main ingredient of \cite{kitazawa0}.
	
	We consider a Morse function ${\tilde{f}}_{S^2,l}$ on a surface ${S^2}_{(l)}$ obtained by removing the interiors of $l>1$ copies of the unit disk $D^2$ smoothly embedded in a $2$-dimensional standard sphere enjoying the following properties. This is a fundamental exercise on Morse functions. See \cite{milnor1, milnor2} for example.
	
	\begin{itemize}
		\item The image is denoted by $[-a,a] \subset \mathbb{R}$ for a real number $a$.   
		\item ${{\tilde{f}}_{S^2,l}}^{-1}(a)$ and one of connected components of the boundary coincide.
		
	\item ${{\tilde{f}}_{S^2,l}}^{-1}(-a)$ and $l-1$ of connected components of the boundary coincide.
	\item Let $\{s_j\}_{j=1}^{l-2}$ denote the set of all singular points of ${\tilde{f}}_{S^2,l}$. For these points, we have ${\tilde{f}}_{S^2,l}(s_j)=-a+2a \frac{1}{l-1}j$.
	\end{itemize} 

We can consider a smooth deformation ${\tilde{f}}_{S^2,l,[-1,1]}:{S^2}_{(l)} \times [-1,1] \rightarrow [-a,a] \times [-1,1]$ of the function enjoying the following properties by fundamental arguments on the theory of Morse functions, smooth functions, and differential topology of manifolds, for example.
	\begin{itemize}
\item Around each singular point of the deformation, it is represented as the product map of a Morse function and the identity map on a line for suitable local coordinates. It is a singular point of a so-called {\it fold} map and the class of fold maps generalizes the class of Morse functions and that of special generic maps. See \cite{golubitskyguillemin} for {\it fold} singularity.
\item The restriction to the singular set of the deformation is a smooth embedding such that around the boundaries our condition on the transversality is satisfied.  
	\item The image is $[-a,a] \times [-1,1]$.   
	\item ${\tilde{f}}_{S^2,l,[-1,1]}((p,-1))={\tilde{f}}_{S^2,l}(p)$.
	\item Outside the union of small neighborhoods of singular points $s_j$, the values of  ${\tilde{f}}_{S^2,l,[-1,1]}$ are constant in the interval $[-1,1]$ in ${S^2}_{(l)} \times [-1,1]$.
	\item In the interval $[-1,1]$ in ${S^2}_{(l)} \times [-1,1]$, the obtained functions are Morse functions whose singular sets are invariant. Furthermore, at $t \in [-1,1]$, $s_j$ is mapped to $((-a+2a \frac{1}{l-1}j)(-t),t) \in [-a,a] \times [-1,1]$. 
\end{itemize} 
 
We restrict ${\tilde{f}}_{S^2,l,[-1,1]}:{S^2}_{(l)} \times [-1,1] \rightarrow [-a,a] \times [-1,1]$ to the preimage of a disk $\{x \in {\mathbb{R}}^2 \mid ||x|| \leq \frac{1}{2}\}$ and compose this with the height function mapping $x$ to $\pm |||x||^2+c$ where $c$ is some real number. We adopt a manifold diffeomorphic to the $3$-dimensional manifold of the domain as $E_p$ and a function ${\tilde{f}}_{p,\partial W_f}$ as the obtained one: more rigorously a smooth function being $C^{\infty}$ equivalent to the obtained function. We construct our desired map as presented in Remark \ref{rem:1} with Proposition \ref{prop:2}. We investigate $\partial E_p$ and its fundamental group.

It is essentially same to investigate ${S^2}_{(l)} \times [-1,1]$ and its boundary. They are diffeomorphic to the original $3$-dimensional manifold $E_p$ and the surface $\partial E_p$ after the corners are eliminated. 
$\partial E_p$ is, by the original argument or fundamental arguments from Morse functions, a closed, connected and orientable surface of genus $l-1$. 
More precisely, we can consider a smooth isotopy from ${S^2}_{(l)} \times [-1,1]$ to $E_p$ by deforming the map in a natural way. We can argue and have smooth maps, $1$-dimensional manifolds diffeomorphic to circles or closed intervals, and other geometric objects as follows.

\begin{itemize}
\item In ${{S^2}_{(l)} \times \{-1\}}$, we have exactly $l-1$ smoothly embedded curves, each of which is denoted by $c_j:[-1,1] \rightarrow {S^2}_{(l)} \times \{-1\}$ for an integer $1 \leq j \leq l-1$. Furthermore, we can have the family enjoying the following properties. Here the connected component ${{\tilde{f}}_{S^2,l,[-1,1]}}^{-1}((a,-1))$ is denoted by $C_0$ and a circle of course. ${{\tilde{f}}_{S^2,l,[-1,1]}}^{-1}((-a,-1))$ consists of exactly $l-1$ disjoint circles, each of which is denoted by $C_j$ for an integer $1 \leq j \leq l-1$. The disjoint union ${\sqcup}_{j=1}^{l-1} C_{j-1}$ and the boundary of the surface ${{S^2}_{(l)} \times \{-1\}}$ coincide of course.
\begin{itemize}
\item $c_j((-1,1)) \subset {\rm Int}\ {{S^2}_{(l)} \times \{-1\}}$.
\item $c_j(-1) \in C_{j-1}$ and $c_j(1) \in C_{j}$.
\item At boundaries, the condition on the transversality is satisfied.
\item The images of distinct curves are mutually disjoint.
\item The images are sufficiently far from the singular set of the function ${\tilde{f}}_{S^2,l,[-1,1]} {\mid}_{{S^2}_{(l)} \times \{-1\}}$.
\end{itemize}
\item We consider the product of the image of each curve before and $[-1,1]$. The boundary is regarded as a circle smoothly embedded in the boundary $\partial\ ({S^2}_{(l)} \times [-1,1])$ and null-homotopic in ${S^2}_{(l)} \times [-1,1]$.
\item The $l-1$ circles in the boundary $\partial\ ({S^2}_{(l)} \times [-1,1])$ are mutually disjoint and (by choosing the curves $c_j$ suitably first) we can obtain a situation such that after cutting $\partial\ ({S^2}_{(l)} \times [-1,1])$ along the circles, we have a compact, connected and orientable surface whose Euler number is $4-2l$ and which has exactly $2(l-1)$ boundary connected components. In terms of \cite{marzantowiczmichalak}, the obtained $l-1$ circles form an {\it independent and regular system of hypersurfaces }({\it with no boundary}) in the closed, connected and orientable surface $\partial\ ({S^2}_{(l)} \times [-1,1])$ of genus $l-1$. Let $\{S_j\}_{j=1}^{l-1}$ denote the family of the smoothly embedded circles.  
\item (By considering suitable situations), we can have another family $\{{S_j}^{\prime}\}_{j=1}^{l-1}$ of disjointly and smoothly embedded circles in $\partial\ ({S^2}_{(l)} \times [-1,1])$ enjoying the following properties due to fundamental topological theory of surfaces.
\begin{itemize}
\item $S_{j_1}$ and ${S_{j_2}}^{\prime}$ do not intersect if $j_1 \neq j_2$ and intersect at some one-point set if $j_1=j_2$.
\item In the one-point set, these distinct circles intersect satisfying a condition on the transversality: the sum of the images of the differentials there is $2$. They give so-called {\it normal-crossings}.
\item There exists a diffeomorphism which preserves the orientation and maps the disjoint union ${\sqcup}_j S_j$ onto ${\sqcup}_j {S_j}^{\prime}$ and the disjoint union ${\sqcup}_j {S_j}^{\prime}$ onto ${\sqcup}_j S_j$.
\item We can choose an arbitrary point in the surface outside the union of these $l-1+l-1=2(l-1)$ circles and connect the point and some point in each of the $2(l-1)$ circles by a smooth curve. For each circle, we can have an element of ${\pi}_1(\partial\ ({S^2}_{(l)} \times [-1,1]))$ and the set of the resulting $2l-2$ elements generates the fundamental group of the surface.   
\end{itemize}
\end{itemize}
FIGURE \ref{fig:1} shows a case for $l=4$. Circles ${S_j}^{\prime}$ are taken in ${{S^2}_{(l)} \times \{-1\}}$ and we can do so in general. They are also chosen as connected components of the preimages of points containing no singular points for the function on ${{S^2}_{(l)} \times \{-1\}}$ and we can do so in general.
\begin{figure}
	
	\includegraphics[height=25mm, width=40mm]{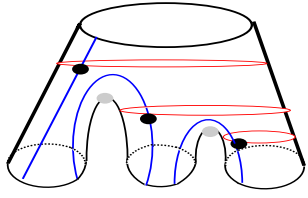}

		\caption{The images of the curves $c_j$ and circles ${S_{j}}^{\prime}$ in the case $l=4$. The images of $c_j$ are depicted in blue and ${S_{j}}^{\prime}$ are in red. Black dots are for the intersection of the curves and circles. Grey dots are for singular points of the function ${\tilde{f}}_{S^2,l,[-1,1]} {\mid}_{{S^2}_{(l)} \times \{-1\}}$.}
		\label{fig:1}

\end{figure}

	Related to these arguments, \cite{marzantowiczmichalak} applies methods closely related to us. This preprint is independent of our study. However methods presented there must be important for us. There, to study homomorphisms from the fundamental group of a compact and connected manifold into free groups, closed intervals or circles disjointly embedded into surfaces, or more generally, such systems of hypersurfaces are studied, for example.  

We go back to our proof. 
Remember that ${S^2}_{(l)} \times [-1,1]$ and $E_p$ are essentially same in our argument. We apply methods in Remark \ref{rem:1}.
By the obtained properties on the surface, its fundamental group and nice systems of circles, for distinct two connected components of the boundary $\partial \bar{N}$, we can choose suitable distinct diffeomorphisms on $\partial \bar{N}$ to have a situation for Theorem \ref{mthm:2}. More precisely, we use the product map of the identity map on the base space and the diffeomorphism on the fiber for each connected component of the boundary $\partial \bar{N}$ to obtain our desired map.

This completes the proof.

\end{proof}
\begin{Ex}
\label{ex:4}
In the case $F:=S^1 \times S^1$ in Theorem \ref{thm:3}, our construction gives simply generalized special generic maps. For example, we consider the situation in the following.
\begin{itemize}
\item A smooth immersion $\bar{f}:\bar{N} \rightarrow N$ of an $n$-dimensional compact and simply-connected manifold $\bar{N}$ is an embedding.
\item $\bar{N}:=S^k \times D^1$ and $N:={\mathbb{R}}^{k+1}$ with $k \geq 2$ and $n=k+1$.  
\end{itemize} 
We easily have a smooth embedding of $\bar{N}$ into $N$ here. For our desired map here, first consider a smooth function on $S^3$ whose image is a closed interval in $\mathbb{R}$ and whose singular set is mapped onto the boundary of the image. Furthermore, the function is chosen as a Morse-Bott function such that the preimage of a point in the interior of the image is diffeomorphic to $F$. We consider the product map of this function and the identity map on $S^k$. After that, we embed the image via the embedding $\bar{f}$ to have our desired map.

We can argue similarly in the case where $F$ is a closed, connected and orientable surface of genus $g>1$ to have a desired map on $S^k \times S^3$. 
\end{Ex}

We explain about another exposition on the proof and ${S^2}_{(l)} \times [-1,1]$ after the earlier version \cite{kitazawa10}.

\begin{Rem}
\label{rem:2}
 Let $l=2$. In this case, ${S^2}_{(l)} \times [-1,1]$ is, after the corner is eliminated, diffeomorphic to $D^2 \times S^1$. Assume that for $l:=l_0 \geq 2$, this is diffeomorphic to a manifold represented as a boundary connected sum of $l_0-1$ copies of $D^2 \times S^1$ taken in the smooth category, after the corner is eliminated. The manifold we consider is obtained by removing the interior of a small regular neighborhood (considered in the smooth category) of a closed interval embedded smoothly and properly in the manifold in the case $l=l_0$. In other words, the boundary of the closed interval is smoothly embedded in the boundary and the interior is embedded in the interior and the condition on the transversality is satisfied. 
	Furthermore, the embedding is taken as a smooth embedding smoothly homotopic to a constant map whose values are always one point in the boundary $\partial {S^{2}}_{(l_0)} \times [-1,1]$ of the manifold ${S^{2}}_{(l_0)} \times [-1,1]$ for the case $l=l_0$. We can also choose the smooth homotopy $e:D^1 \times [0,1] \rightarrow {S^{2}}_{(l_0)} \times [0,1]$ so that the restriction to $D^1 \times [0,1)$ and that to each $D^1 \times \{t^{\prime}\}$ for $t^{\prime} \in [0,1)$ are smooth embeddings. In other words, this explains about the property that the embedding of the closed interval is "unknotted" in the smooth category.
	We can understand that our manifold is diffeomorphic to a manifold represented as a boundary connected sum of $l_0=l-1$ copies of $D^2 \times S^1$ taken in the smooth category.
\end{Rem}

In addition, our arguments also yield the following.

\begin{MainThm}
\label{mthm:3}
We can extend Theorem \ref{thm:3} to the case satisfying the following naturally generalized conditions.
\begin{enumerate}
\item $m=n+k_0$ where $k_0>2$ is an integer.
\item $F$ is a smooth closed and connected manifold represented as a connected sum of finitely many copies of $S^1 \times S^{k_0-1}$ taken in the smooth category and not a sphere.
\end{enumerate}
\end{MainThm}

We present another phenomenon on SGL maps. 
We explain about simply generalized special generic maps in Example \ref{ex:2} (\ref{ex:2.2}) and Remark \ref{rem:1} again. In our arguments here, as the coefficient ring $A$, we take a principal ideal domain having the unique identity element different from the zero element. According to our main work of \cite{kitazawa9}, as Remark \ref{rem:1} presents shortly, we can easily construct a natural simply generalized special generic map on some suitable $m$-dimensional closed and connected manifold $M$ such that in Example \ref{ex:2} (\ref{ex:2.2}), for an arbitrary integer $1 \leq j \leq l+1$, we can choose $j_p:=j$ for some point $p \in \partial \bar{N}$ (if the boundary of the manifold $\bar{N}$ has sufficiently many connected components). According to our construction, the resulting reference homomorphisms along the curves between the homology groups of degrees at least $1$ are always the zero homomorphisms. This is also due to the structure of the (co)homology group of the product of spheres, understood by so-called K\"unneth theorem for example. See \cite{hatcher} again. 
K\"unneth theorem also shows the following.
\begin{itemize}
\item We can choose the cohomology duals to the elements of the homology groups being also the generators of the kernels of the presented reference homomorphisms along the curves and represented by the spheres in the product of the spheres, identified with $F$, in Example \ref{ex:2} (\ref{ex:2.2}), respecting some suitable basis. For this, universal coefficient theorem is important for example. See \cite{hatcher} again. 
\item The set of these cohomology duals respecting the basis generates the cohomology ring of $F$.
\end{itemize} 
In other words, we may expect the following phenomenon: these two conditions on the reference homomorphisms along the curves and the structure of the cohomology ring of $F$ induce the fact that the reference homomorphisms along the curves between the homology groups of degrees at least $1$ are always the zero homomorphisms.

\begin{Prob}
\label{prob:3}
Are phenomena as presented here observed in general? 
\end{Prob}
\begin{Ex}
\label{ex:5}
Maps in Theorem \ref{thm:3} and Main Theorem \ref{mthm:3} also support Problem \ref{prob:3}.
\end{Ex}
We give a counterexample.
\begin{MainThm}
	\label{mthm:4}
The answer to Problem \ref{prob:3} is false in general.	
\end{MainThm}
\begin{proof}
Let $k>1$ be an integer. Let $A:=\mathbb{Z}/k\mathbb{Z}$, which is the naturally defined quotient ring of the ring $\mathbb{Z}$ and also a cyclic group of order $k$ under the products for the quotient ring.
We need the notions of the {\it Euler number} and, more generally, the {\it Euler class} of a linear bundle and for them, see \cite{milnorstasheff} for example. 
This explains about classical theory of linear bundles systematically and essentially same notions are presented.
In our case or the case of linear bundles over closed, connected and orientable surfaces whose fibers are $S^1$ or $D^2$, Euler numbers and Euler classes are essentially same by fundamental algebraic topology such as Poincar\'e duality for the surfaces.

We can consider a linear bundle ${\tilde{M}}^{4,k}$ over $S^2$ whose fiber is the unit disk $D^2$ and whose subbundle $\partial {\tilde{M}}^{4,k}$ obtained by restricting the fiber to $\partial D^2 \subset D^2$ is of Euler number $k$.
By some fundamental theory on construction of special genric maps of \cite{saeki1}, we have a smooth map ${\tilde{f}}_{4,k}: {\tilde{M}}^{4,k} \rightarrow S^2 \times \mathbb{R}$ enjoying the following properties.
\begin{itemize}
\item The image is denoted by $S^2 \times [a,b] \subset S^2 \times \mathbb{R}$.
\item The singular set is mapped onto $S^2 \times \{b\}$.
\item The preimage of $S^2 \times \{a\}$ and the boundary coincide.
\item Around each point of $S^2 \times \{b\}$, it is locally represented as the product map of a height function on a copy of the unit disk $D^2$ and the identity map on a $2$-dimensional disk.  More rigorously, on an open neighborhood $U \subset S^2$ diffeomorphic to the interior
 of the unit disk $D^2$, the restriction of ${\tilde{f}}_{4,k}$ to the preimage of $U \times [a,b]$, we have some diffeomorphism ${\Phi}_U:{{\tilde{f}}_{4,k}}^{-1}(U \times [a,b]) \rightarrow U \times D^2$ and the relation 
$({\rm id}_{U} \times {\tilde{h}}_{2, [a,b]}) \circ {\Phi}_U={\tilde{f}}_{4,k}: {{\tilde{f}}_{4,k}}^{-1}(U \times [a,b]) \rightarrow U \times [a,b]$ where the notation is as follows.
\begin{itemize}
\item ${\rm id}_{U}$ is the identity map on $U$. 
\item ${\tilde{h}}_{2,[a,b]}$ is a (suitably scaled) height function on (a copy of) $D^2$.
\end{itemize}
\end{itemize}
An argument essentially same as this is also important in Theorem 5.7 (3) of our preprint \cite{kitazawasaeki} for example. 

By composing the resulting map into $S^2 \times \mathbb{R}$ with the projection to $\mathbb{R}$ gives a function suitable for ${\tilde{f}}_{p, \partial W_f}$ in Definition \ref{def:3}. In Definition \ref{def:3}, $E_p={\tilde{M}}^{4,k}$ and $F$ is diffeomorphic to $\partial {\tilde{M}}^{4,k}$ of course. We can construct an SGL map as in Remark \ref{rem:1}.

We have the homology exact sequence for the pair $({\tilde{M}}^{4,k}, \partial {\tilde{M}}^{4,k})$. See \cite{hatcher} again for the sequence. By using some theory such as homology groups of closed, connected and orientable manifolds here and Poincar\'e duality, we can find isomorphisms, represented via "$\cong$", in the sequence
$$\rightarrow H_3(\partial {\tilde{M}}^{4,k};A) \cong A \rightarrow H_3( {\tilde{M}}^{4,k};A) \cong \{0\} \rightarrow H_3({\tilde{M}}^{4,k},\partial {\tilde{M}}^{4,k};A) \cong  H^1({\tilde{M}}^{4,k};A) \cong \{0\} $$ \\
$$\rightarrow H_2(\partial {\tilde{M}}^{4,k};A) \cong A \rightarrow H_2( {\tilde{M}}^{4,k};A) \cong A \rightarrow H_2({\tilde{M}}^{4,k},\partial {\tilde{M}}^{4,k};A) \cong H^2({\tilde{M}}^{4,k};A) \cong A$$ \\
$$\rightarrow H_1(\partial {\tilde{M}}^{4,k};A) \cong A \rightarrow H_1( {\tilde{M}}^{4,k};A) \cong \{0\} $$ \\
and the inclusion induces an isomorphism between $H_2(\partial {\tilde{M}}^{4,k};A)$ and $H_2( {\tilde{M}}^{4,k};A)$. 

The cohomology ring $H^{\ast}(\partial {\tilde{M}}^{4,k};A)$ of $\partial {\tilde{M}}^{4,k}$ is known to be generated by an element of $H^{1}(\partial {\tilde{M}}^{4,k};A)$. Furthermore, $H_{j}(\partial {\tilde{M}}^{4,k};A)$ and $H^{j}(\partial {\tilde{M}}^{4,k};A)$ 
are known to be free for $j=0, 1, 2, 3$ and of rank $1$.
We can see that this completes the proof.
\end{proof}
\begin{Rem}
\label{rem:3}
	In our maps in our proof of Main Theorem \ref{mthm:4}, ${\pi}_1(F)$ is also known to be isomorphic to $A$. So we easily have cases for Theorem \ref{thm:2} if other objects are suitably given where "$k=1$ in Theorem \ref{thm:2}".
\end{Rem}
\begin{Ex}
\label{ex:6}
In the case $F:=\partial {\tilde{M}}^{4,k}$ in Main Theorem \ref{mthm:4}, we consider the situations in the following two. In each situation, our smooth immersion $\bar{f}:\bar{N} \rightarrow N$ of an $n$-dimensional compact and simply-connected manifold $\bar{N}$ can be chosen as a suitable embedding.
\begin{enumerate}
\item Let $\bar{N}:=D^n$ and $N:={\mathbb{R}}^{n}$. We easily have a smooth embedding $\bar{f}$ of $\bar{N}$ into $N$ here. For our desired map here, first we can consider a smooth function on the manifold ${\tilde{M}}^{4,k}$ whose image is a closed interval in $\mathbb{R}$ and whose singular set is mapped onto the boundary of the image. 
We respect ${\tilde{f}}_{p, \partial W_f}$ in the proof of Main Theorem \ref{mthm:4} to construct this function and we can do this. We can construct as in Remark \ref{rem:1} and we can do in such a way that a section of the linear bundle over $S^2$, more precisely, an element represented by this, can give a generator of $H_2( {\tilde{M}}^{4,k};A) \cong A$ in the proof.
\item
Let $\bar{N}:=S^{k_0} \times D^1$ and $N:={\mathbb{R}}^{k_0+1}$ where $k_0 \geq 2$ is an integer with $n=k_0+1$. This is similar to Example \ref{ex:4}. We easily have a smooth embedding $\bar{f}$ of $\bar{N}$ into $N$ here. For our desired map here, first we can consider a smooth function on the manifold ${\tilde{M}}^{4,k}$ whose image is a closed interval in $\mathbb{R}$ and whose singular set is mapped onto the boundary of the image as before. 
We respect ${\tilde{f}}_{p, \partial W_f}$ in the proof of Main Theorem \ref{mthm:4} to construct this function and we can do this as before. We consider the product map of this function and the identity map on $S^{k_0}$. After that, we embed the image via the embedding $\bar{f}$ to have our desired map.

We consider natural construction here. In this construction, the manifold of the domain can be determined by using some fundamental theory of linear bundles. In the case $k$ is even, the manifold of the domain is diffeomorphic to the product of $S^{k_0}$ and $S^2 \times S^2$. In the case $k$ is odd, the manifold of the domain is diffeomorphic to the product of $S^{k_0}$ and the total space of a linear bundle over $S^2$ whose fiber is the unit sphere $S^2$ and which is not trivial. Sections of the linear bundles over $S^2$, more precisely, elements represented by the sections, can give generators of $H_2( {\tilde{M}}^{4,k};A) \cong A$ in the proof.
\end{enumerate}
\end{Ex}
\begin{MainThm}
	\label{mthm:5}
	For SGL maps constructed in Theorem \ref{thm:3} under the constraint that the 2nd integral homology group $H_2(\bar{N};\mathbb{Z})$ is free and Example \ref{ex:6} for example, the 2nd integral
 homology group $H_2(M;\mathbb{Z})$ of the closed and simply-connected manifold $M$ of the domain is free.

More generally, we consider our construction of Theorem \ref{thm:2} satisfying the following conditions in the case $k=1$ there. This situation generalizes the explicit situations above.
\begin{enumerate}
\item
\label{mthm:5.1}
$F$ is connected and the 2nd homotopy group ${\pi}_2(F)$ is the trivial group. 
\item
\label{mthm:5.2}
 For any point $p \in \partial W_f$, ${q_f}^{-1}(p)$ is connected and at most $k^{\prime}$-dimensional with $k^{\prime}<m-n$.
\item
\label{mthm:5.3}
The 2nd integral homology group $H_2(W_f;\mathbb{Z})$ is free.
\end{enumerate}
Then the 2nd integral
 homology group $H_2(M;\mathbb{Z})$ of the closed and simply-connected manifold $M$ of the domain is free.
	\end{MainThm}

\begin{proof} 
Our construction of Theorem \ref{thm:3} and Example \ref{ex:6} shows that the condition (\ref{mthm:5.2}) holds. For our construction in Theorem \ref{thm:3} under
 the constraint that the 2nd integral homology group $H_2(\bar{N};\mathbb{Z})$ is free and Example \ref{ex:6}, we can easily see that the conditions (\ref{mthm:5.1}) and  (\ref{mthm:5.3}) hold.

Assume that the 2nd integral
 homology group $H_2(M;\mathbb{Z})$ of the closed and simply-connected manifold $M$ of the domain is not free. We can take an element which is not the zero element and which is of a finite order. This is represented by
 a $2$-dimensional sphere $S_0$ smoothly embedded in the total space ${q_f}^{-1}(W_f-{\rm Int}\ N(\partial W_f))$ of a smooth bundle over $W_f-{\rm Int}\ N(\partial W_f)$ in Definition \ref{def:3} by Hurewicz theorem. This is due to the fact that the dimension of the singular set is calculated as $k^{\prime}+(n-1)=(m-n-1)+n-1=m-2$, followed from the condition (\ref{mthm:5.2}).
We apply the homotopy exact sequence for the bundle ${q_f}^{-1}(W_f-{\rm Int}\ N(\partial W_f))$ over $W_f-{\rm Int}\ N(\partial W_f)$ whose fiber is diffeomorphic to $F$ with the conditions (\ref{mthm:5.1}) and (\ref{mthm:5.3}).
We can see that the 2nd homotopy group of ${q_f}^{-1}(W_f-{\rm Int}\ N(\partial W_f))$ is free and isomorphic to that of $W_f-{\rm Int}\ N(\partial W_f)$. 
This isomorphism is induced from the projection of the bundle. This contradicts the condition on the homotopy on $S_0$ and $M$.

This completes the proof.
\end{proof}

According to \cite{nishioka}, a compact and connected manifold $X$ whose dimension is greater than $3$, whose boundary is not empty and whose 1st homology group $H_1(X;A)$ is orientable and the $j$-th homology group $H_{j}(X;A)$ is free for $j \geq \dim X-2$ for any commutative ring $A$ where the homology groups are seen as modules over $A$. In the case $\dim X \leq 4$ here, $H_j(X;A)$ is free as a module over $A$ for $j \geq 0$. See Lemma 3.1 of \cite{nishioka}.
This is used in determining $5$-dimensional closed and simply-connected manifolds admitting special generic maps into ${\mathbb{R}}^4$ by applying the case $\dim X=4$ and the classification of these manifolds in \cite{barden}.
Nishioka's answer is same as Theorem \ref{thm:1} (\ref{thm:1.2}) shows. Barden's answer shows that classifications of $5$-dimensional closed and simply-connected manifolds are same in the smooth category, the PL category (piecewise smooth category), and the smooth category. 

We go back to our main arguments. By our arguments, we have the following. We can check easily.

\begin{Cor}
In Main Theorem \ref{mthm:5}, the condition {\rm (}\ref{mthm:5.3}{\rm )} automatically holds in the case the dimension $n$ is smaller than or equal to $4$ in the situation of Theorem \ref{thm:2}.
\end{Cor}
Our last result is on singularity of resulting SGL maps. We need the notion of a {\it fold} map, presented shortly in the proof of Theorem \ref{thm:3}.
\begin{MainThm}
	\label{mthm:6}
	For SGL maps constructed in Theorem \ref{thm:3} and Main Theorem \ref{mthm:4}, if we restrict the maps to small open neighborhoods of singular points, then the resulting maps are represented as the compositions of fold maps. Here, projections, height functions, and Morse functions, are also fold maps. 
	\end{MainThm}
\section{Acknowledgement}
The author would like to thank Takahiro Yamamoto for related discussions on our previous study \cite{kitazawa9}, especially, on the terminology "generalized special generic maps" and meanings of our studies in the singularity theory of differentiable maps and applications to geometry of manifolds. These discussions have motivated the author to study further and contributed to our present study.	
	
	\end{document}